\documentclass[10pt]{amsart}

\usepackage{amsmath,amsthm,amssymb}
\usepackage{color}

\begin{document}

\newtheorem{proposition}{Proposition}[section]
\newtheorem{definition}[proposition]{Definition}
\newtheorem{corollary}[proposition]{Corollary}
\newtheorem{theorem}[proposition]{Theorem}
\newtheorem{condition}{Condition}
\newtheorem{remark}{Remark}[section]

\newcommand{\al}{\alpha} 
\newcommand{\bt}{\beta} 
\newcommand{\de}{\delta} 
\newcommand{\ga}{\gamma} 
\newcommand{\G}{\Gamma} 
\newcommand{\e}{\epsilon} 
\newcommand{\la}{\lambda} 
\newcommand{\si}{\sigma} 
\newcommand{\ph}{\varphi} 
\newcommand{\pa}{\partial}       
\newcommand{\vn}{\varnothing}       
\newcommand{\coi}{C_0^{\infty}} 
\newcommand{\ci}{C^{\infty}} 
\newcommand{\ts}{\text{supp}\,} 
\newcommand{\tss}{\text{singsupp}\,} 
\newcommand{\br}{{\mathbb R}} 
\newcommand{\rn}{\mathbb R^n} 
\newcommand{\adx}{\vct\cdot x} 
\newcommand{\Sn}{S^{n-1}} 
\newcommand{\sg}{\text{sgn}} 
\newcommand{\ibr}{\int_{\br}} 
\newcommand{\brt}{\br^2}       
\newcommand{\irt}{\int_{\brt}}       
\newcommand{\ioi}{\int_0^{\infty}}       
\newcommand{\iii}{\int_{-\infty}^{\infty}}       
\newcommand{\isn}{\int_{\Sn}}       
\newcommand{\fle}{f_{\Lambda\e}}
\newcommand{\CS}{{\mathcal S}}
\newcommand{\CD}{{\mathcal D}}

\newcommand{\CB}{\mathcal B} 
\newcommand{\CE}{\mathcal E} 

\numberwithin{equation}{section}

\title[Inversion of restricted ray transforms]{Reconstruction algorithms for a class of restricted ray transforms without added singularities}
\author{A. Katsevich}
\address{Department of Mathematics, University of Central Florida, Orlando, FL 32816-1364}
\email{alexander.katsevich@ucf.edu}
\thanks{This work was supported in part by NSF grants DMS-1115615 and DMS-1211164}

\date{}

\begin{abstract}
Let $X$ and $X^*$ denote a restricted ray transform along curves and a corresponding backprojection operator, respectively. Theoretical analysis of reconstruction from the data $Xf$ is usually based on a study of the composition $X^* D X$, where $D$ is some local operator (usually a derivative). If $X^*$ is chosen appropriately, then $X^* D X$ is a Fourier Integral Operator (FIO) with singular symbol. The singularity of the symbol leads to the appearance of artifacts (added singularities) that can be as strong as the original (or, useful) singularities. By choosing $D$ in a special way one can reduce the strength of added singularities, but it is impossible to get rid of them completely.  

In the paper we follow a similar approach, but make two changes. First, we replace $D$ with a nonlocal operator $\tilde D$ that integrates $Xf$ along a curve in the data space. The result $\tilde D Xf$ resembles the generalized Radon transform $R$ of $f$. The function $\tilde D Xf$ is defined on pairs $(x_0,\Theta)\in U\times S^2$, where $U\subset\br^3$ is an open set containing the support of $f$, and $S^2$ is the unit sphere in $\br^3$. Second, we replace $X^*$ with a backprojection operator $R^*$ that integrates with respect to $\Theta$ over $S^2$. It turns out that if $\tilde D$ and $R^*$ are appropriately selected, then the composition $R^* \tilde D X$ is an elliptic pseudodifferential operator of order zero with principal symbol 1. Thus, we obtain an approximate reconstruction formula that recovers all the singularities correctly and does not produce artifacts. The advantage of our approach is that by inserting $\tilde D$ we get access to the frequency variable $\Theta$. In particular, we can incorporate suitable cut-offs in $R^*$ to eliminate bad directions $\Theta$, which lead to added singularities. 
\end{abstract}


\maketitle

\section{Introduction} \label{sec:intro}

Problems where a function, a vector field, or a tensor field needs to be reconstructed from its integrals along a family of curves occur in many applications, such as medical computed tomography (CT), geophysics, doppler tomography, electron microscopy, etc. (see e.g. \cite{bkr, uhl-01, sch-08, qr-13, hq15} and references therein). In this paper we consider a particular version of the problem, which is inspired by medical applications of CT, when the object being scanned undergoes a deformation (or, moves) during the scan. The most common example is cardiac CT. Even the fastest scanners available on the market today do not allow one to completely ``freeze" the motion of the heart, which leads to motion artifacts in the reconstructed images. See \cite{bkr} for an overview of the different concepts used in dynamic CT. From the mathematical perspective, the data in dynamic CT consists of integrals of the unknown attenuation coefficient $f(t,\cdot)$ along lines intersecting a curve in space. The latter is usually called the source trajectory, and the corresponding integral transform is called the restricted ray transform. Since the object changes during the scan, integrals of $f(t,\cdot)$ along lines at any time $t$ correspond to integrals of $f(t_0,\cdot)$ along some curves at reference time $t=t_0$.

While in certain cases deformations can be compensated theoretically exactly (see \cite{drg-07}), there is no exact inversion formula that can handle general motions that are practically relevant. Let $X$ and $X^*$ denote a restricted ray transform along curves and a corresponding backprojection operator, respectively. 
In the absence of inversion formulas, theoretical analysis of reconstruction from the data $Xf$ is usually based on the study of the composition $X^* D X$, where $D$ is some local (e.g., differential) operator \cite{gruhl, kat7, flu, qr-13, kq-15}. 
Usually $X^*$ is related to the formal dual of $X$, and may contain various cut-offs to make sure the composition $X^* D X$ is well-defined.
If $X^*$ is chosen appropriately, then $X^* D X$ is a Fourier Integral Operator (FIO) with singular symbol \cite{gruhl, flu}. The singularity of the symbol leads to the appearance of artifacts (added singularities) that can be as strong as the original (or, useful) singularities \cite{kat7, flu}. By choosing $D$ in a special way one can reduce the strength of added singularities \cite{kat06a}, but it is impossible to get rid of them completely. See also \cite{fq-11, qr-13} for applications of a similar idea in other settings. In the case of static objects, operators of the type $X^* D X$ are closely related to local (or, Lambda) tomography \cite{luma, rk, fbh01}. 

In the paper we follow a similar approach, but make two changes. First, we replace $D$ with a nonlocal operator $\tilde D$ that integrates $Xf$ along a curve in the data domain. The result $\tilde D Xf$ resembles the generalized Radon transform $R$ of $f$. The function $\tilde D Xf$ is defined on pairs $(x_0,\Theta)\in U\times S^2$, where $U\subset\br^3$ is an open set containing the support of $f$, and $S^2$ is the unit sphere in $\br^3$. Second, we replace $X^*$ with a backprojection operator $R^*$ that integrates with respect to $\Theta$ over $S^2$. It turns out that if (a) the source trajectory and the deformation of the object satisfy certain conditions, and (b) $\tilde D$ and $R^*$ are appropriately selected, then the composition $R^* \tilde D X$ is an elliptic pseudodifferential operator (PDO) of order zero with principal symbol 1. Thus, we obtain an approximate reconstruction formula that recovers all the singularities correctly and does not produce artifacts. The advantage of our approach is that by inserting $\tilde D$ we get access to the frequency variable $\Theta$. In particular, we can incorporate suitable cut-offs in $R^*$ to eliminate undesirable directions $\Theta$, which lead to added singularities. Such control is impossible when one uses operators of the type $X^* D X$, where $D$ is a differential operator. 

It is worth mentioning an important related paper \cite{fsu-08}, where inversion of a fairly general class of ray transforms is studied. Our results are different, because in \cite{fsu-08} the problem is overdetermined. If the dimension of the space is $n=3$, the data in \cite{fsu-08} has four degrees of freedom. In our case the data has three degrees of freedom. 

The paper is organized as follows. In Section~\ref{sec:prelim} we introduce the main notations and describe the reconstruction problem. In Section~\ref{sec:first_inv} we formulate the assumptions about the source trajectory and the deformation of the object. We also construct the operators $R^*$ and $\tilde D$ such that the composition $R^*\tilde D$ inverts the ray transform up to the leading order. In particular, we show that $\CB:=R^*\tilde D X$ is an elliptic PDO of order 1. 
In Section~\ref{sec:anal_inv} we consider a family of deformations depending on a parameter $\e$. We prove that if the deformation of the object becomes small as $\e\to0$, then $\CB_\e-\text{Id}\to0$. Here $\text{Id}$ is the identity operator, and the difference $\CB_\e-\text{Id}$ is viewed as an operator $H^\nu_0(U)\to H^{\nu-1}_{loc}(U)$. This result is similar to the one obtained in \cite{kat10b} in the case of a two-dimensional dynamic reconstruction problem. 
In Section~\ref{sec:alt_inv} we consider the static case and construct a localized operator $\tilde D$ such that computing $\CB f=R^*\tilde D X f$ at any $x_0\in U$ uses integrals of $f$ along lines passing through a small neighborhood of $x_0$. The neighborhood can be made as small as one likes, but $\CB$ is still an elliptic PDO with (complete) symbol 1.

In Section~\ref{sec:general} we compare the reconstructions based on $R^*\tilde D X$ and $X^*X$. It turns out that if no cut-offs are used in $R^*$, then both operators add singularities in the same places. This implies that using $R^*\tilde D$ instead of $X^*$ does not alter the nature of reconstruction from the restricted ray transform in a fundamental way. Instead, it allows one to use redundancies in the data to suppress artifacts. We also describe several generalizations of the algorithms of Sections~\ref{sec:first_inv} and \ref{sec:alt_inv}. In particular, we briefly outline other algorithms that can be useful for various applications. The algorithms of Sections~\ref{sec:first_inv} and \ref{sec:alt_inv} have been singled out and described in more detail because of the following two reasons. First, they illustrate the main ideas of the paper. Second, they have some special properties. The one in Section~\ref{sec:first_inv} is proven to converge to the exact inversion formula in a fairly strong sense if the deformation becomes small. The one in Section~\ref{sec:alt_inv} uses only local data and inverts the ray transform up to a $\ci$ function. These properties are important from the practical perspective. Finally, the proof of a technical result is presented in Appendix~\ref{sec:app}.

\section{Preliminaries} \label{sec:prelim}

Let $C$ be a piecewise smooth, non-selfintersecting curve in $\br^3$
\begin{equation}\label{curve}
\bigcup_{k=1}^K (a_k,b_k)=:I\ni s\to z(s)\in\br^3,\ |d z(s)/ds|\not=0,
\end{equation}
where $-\infty<a_k<b_k<\infty$, $1\leq k\leq K$,  the intervals $(a_k,b_k)$ are disjoint, and $\sup_{s\in I}|d z(s)/ds|<\infty$. Usually the source moves along (each segment of) $C$ with constant speed, so we identify $s$ with time variable. 

Fix any $s_0\in I$. We refer to $s=s_0$ as the reference time. To describe the deformation of the object being scanned, we introduce a function $\psi$. If at reference time $s_0$ a particle is located at the point $x$, then at time $s$ it is located at the point $y=\psi(s,x)$. We assume that for each $s\in I$ the function $\psi(s,x):\br^3\to \br^3$ is a diffeomorphism. Physically this means that two distinct points cannot move into the same position. This assumption is quite natural, since deformations of objects are not infinitely compressible. The inverse of $\psi$ is the function $x=\nu(s,y):\br^3\to \br^3$. If at time $s$ a particle is located at the point $y$, then $x=\nu(s,y)$ is the position of the particle at the reference time. We assume that (i) $C$ is at a positive distance from an open, bounded set $U$, which contains the support of the object for all $s\in I$, (ii) $\psi,\nu\in\ci(I\times \br^3)$, and (iii) $\psi$ and $\nu$ are the identity maps outside of $U$. 

Since matter is conserved, the x-ray attenuation coefficient at time $s$ and point $y$ is given by $|\pa_y\nu(s,y)| f(\nu(s,y))$. Here we assumed that the x-ray attenuation coefficient of the object is proportional to the density of the object. To account for more general dependence of the attenuation coefficient on density we introduce another factor $A(s,x)$, which is supposed to be a $\ci(I\times U)$ function, positive, bounded away from zero, and known. Hence, the attenuation coefficient of the object is represented by the function
\begin{equation}\label{object}
f_s(y):=A(s,\nu(s,y))|\pa_y\nu(s,y)|f(\nu(s,y)),\ s\in I.
\end{equation}
Consequently, the tomographic data are
\begin{equation}\label{data}
X_{f_s}(\bt):=\ioi f_s(z(s)+t\bt)dt,\ s\in I,\bt\in S^2.
\end{equation}

\section{First approximate inversion formula} \label{sec:first_inv}

The main idea of the derivation in this section is to apply the Grangeat formula to the ray transform data $X_{f_s}$ to obtain a function $Q$ that resembles the first derivative of the generalized Radon transform of $f$. Then, application of a suitably adopted Radon transform inversion formula to $Q$ will produce a reconstruction formula with the desired properties.

Applying the Grangeat formula to $X_{f_s}$ (or the identity in \cite{hssw}) gives
\begin{equation}\label{grang}
-\left.\pa_p\hat f_s(\al,p)\right|_{p=\al\cdot z(s)}=\int_{S^2} X_{f_s}(z(s),\bt)\delta'(\al\cdot\bt)d\bt,
\end{equation}
where $\hat f_s$ is the Radon transform of $f_s$. Using \eqref{object} rewrite the left side of \eqref{grang}:
\begin{equation}\label{rtder}
\begin{split}
-\left.\pa_p\hat f_s(\al,p)\right|_{p=\al\cdot z(s)}&=\int_{\br^3} f_s(y)\delta'(\al\cdot(y-z(s)))dy\\
&=\int_{\br^3} A(s,\nu(s,y))|\pa_y\nu(s,y)|f(\nu(s,y))\delta'(\al\cdot(y-z(s)))dy\\
&=\int_{\br^3} A(s,x)f(x)\delta'(\al\cdot(\psi(s,x)-z(s)))dx.
\end{split}
\end{equation}
Here we have used that $\nu(s,\cdot)$ and $\psi(s,\cdot)$ are the inverses of each other. 
Denote
\begin{equation}\label{beta-def}
\beta(s,x):=\frac{\psi(s,x)-z(s)}{|\psi(s,x)-z(s)|},\ x\in U.
\end{equation}
Let $\ga_{s,x}(t)$ be the preimage of the ray $z(s)+t(\psi(s,x)-z(s))$, $t>0$, at reference time, i.e. $\ga_{s,x}(t):=\nu(s,z(s)+t(\psi(s,x)-z(s)))$. Alternatively, intersection of the curve with $U$ can be described as follows: $\ga_{s,x}=\{w\in U:\bt(s,w)=\bt(s,x)\}$. By construction, $\ga_{s,x}(1)=x$  for all $s\in I$. As is easy to check, the vector  $d_x\psi(s,x)^{-1}\bt(s,x)$ is tangent to $\ga_{s,x}$ at $x$. Therefore, we introduce the notation:
\begin{equation}\label{gadot-def}
\dot\ga(s,x):=d_x\psi(s,x)^{-1}\bt(s,x).
\end{equation}

Let $x_0\in U$ be a reconstruction point. Given $s\in I$, choose any $\al\in S^2$ such that 
\begin{equation}\label{cond}
\al\cdot \bt(s,x_0)=0.
\end{equation}
With $\al$ satisfying \eqref{cond}, the argument of the delta-function in \eqref{rtder} becomes
\begin{equation}\label{theta}
\begin{split}
\al\cdot(\psi(s,x)-z(s))&=\al\cdot(\psi(s,x)-\psi(s,x_0))\\
&=\al\cdot d_x\psi(s,x_0)(x-x_0)+O(|x-x_0|^2)\\
&=\Theta_1\cdot(x-x_0)+O(|x-x_0|^2),\\
\Theta_1:&=d_x\psi(s,x_0)^T\al.
\end{split}
\end{equation}
Next we fix $\Theta\in S^2$ and solve the following equation for $s$ (cf. \eqref{cond}):
\begin{equation}\label{s-eqn}
\Theta\cdot \dot\ga(s,x_0)=0.
\end{equation}
This way we obtain several local solutions $s=s_j(x_0,\Theta)$. In view of \eqref{theta}, we use these local solutions to define
\begin{equation}\label{al-fam}
\al=\al_j(x_0,\Theta):=d_x\psi(s_j(x_0,\Theta),x_0)^{-T}\Theta.
\end{equation}
Here $\al$ is not necessarily a unit vector. For some pairs $(x_0,\Theta)\in U\times S^2$ there can be infinitely (even uncountably) many local solutions to \eqref{s-eqn}. Later we will select only a finite subset of these solutions.

Equation \eqref{s-eqn} implies that the curve $\{x\in U:\bt(s_j,x)=\bt(s_j,x_0)\}$ is perpendicular to $\Theta$ at $x=x_0$.

Divide the right-hand side of \eqref{rtder} by $A(s,x_0)$ and denote
\begin{equation}\label{integral}
Q_j(x_0,\Theta):=\int_{\br^3} \frac{A(s,x)}{A(s,x_0)}f(x)\delta'(\al\cdot(\psi(s,x)-z(s)))dx,\ 
\al=\al_j(x_0,\Theta),s=s_j(x_0,\Theta). 
\end{equation}
Then we denote 
\begin{equation}\label{t_vars}
x_t:=x_0+t\Theta,\ \al_t:=\al_j(x_t,\Theta),\ s_t:=s_j(x_t,\Theta).
\end{equation} 
For simplicity, the subscript $j$ is omitted from $\al_t,s_t$. Replacing $x_0$ with $x_t$ in \eqref{integral} and differentiating with respect to $t$ gives
\begin{equation}\label{der-int}
\begin{split}
\frac{d}{d t}&Q_j(x_t,\Theta)\biggr|_{t=0}\\
=&\int_{\br^3} f(x) \frac{\pa}{\pa s}\left(\frac{A(s,x)}{A(s,x_0)}\right)\left.\frac{d s_t}{d t}\right|_{t=0} \delta'(\al\cdot(\psi(s,x)-z(s)))dx\\
&+\int_{\br^3} f(x) \frac{A(s,x)}{A(s,x_0)}\left.\frac{d}{d t}\left[\al_t\cdot(\psi(s_t,x)-z(s_t))\right]\right|_{t=0} \delta''(\al\cdot(\psi(s,x)-z(s)))dx,\\
\al&=\al_{t=0}=\al_j(x_0,\Theta),\ s=s_{t=0}=s_j(x_0,\Theta).
\end{split}
\end{equation}
Using \eqref{t_vars} we compute
\begin{equation}\label{der-c1}
\begin{split}
\left.\frac{d}{d t}\left[\al_t\cdot(\psi(s_t,x)-z(s_t))\right]\right|_{t=0} 
&=\left.\frac{d}{d t}\left[\al_t\cdot(\psi(s_t,x)-\psi(s_t,x_t))\right]\right|_{t=0}\\
&=\left.\frac{d}{d t}\left[\al_t\cdot\left\{d_x\psi(s_t,x_t)(x-x_t)+O(|x-x_t|^2)\right\}\right]\right|_{t=0}\\
&=\left.\frac{d}{d t}\left[\Theta\cdot(x-x_t)+O(|x-x_t|^2)\right]\right|_{t=0}\\
&=-1+O(|x-x_0|).
\end{split}
\end{equation}
Additionally,
\begin{equation}\label{A-ratio}
\frac{A(s,x)}{A(s,x_0)}=1+O(|x-x_0|).
\end{equation}
In \eqref{der-c1} we need to make sure that the term $O(|x-x_0|)$ is uniform over the relevant range of parameters. By construction, the only term that can blow up is $\left.d s_t/d t\right|_{t=0}$. To compute this derivative we substitute $s=s_t$ and $x_0=x_t$ in \eqref{s-eqn} and write it in the form
\begin{equation}\label{s-eqn-1}
\Theta\cdot \dot\ga(s_t,x_t)\equiv 0.
\end{equation}
Differentiating \eqref{s-eqn-1} and setting $t=0$ yields:
\begin{equation}\label{dsdt}
\left.\frac{d s_t}{d t}\right|_{t=0}=\left.
-\frac{\Theta\cdot \{d_x\dot\ga(s,x)\Theta\}}
{\Theta\cdot \pa_s\dot\ga(s,x_0)}\right|_{s=s_j(x_0,\Theta)},
\end{equation}
provided that
\begin{equation}\label{denom}
\Theta\cdot \pa_s\dot\ga(s_j(x_0,\Theta),x_0)\not=0.
\end{equation}

Condition \eqref{denom} means that the vector tangent to the curve $\{x\in U:\bt(s,x)=\bt(s,x_0)\}$ at $x=x_0$ does not stay in the plane $\Pi(x_0,\Theta):=\{x\in\br^3:\,(x-x_0)\cdot\Theta=0\}$ when $s$ changes infinitesimally in a neighborhood of $s=s_j$. 
Equations \eqref{s-eqn} and \eqref{denom} are analogous to the Kirillov-Tuy condition in the static case \cite{kir, tuy}. The condition says that every plane passing through the object support intersects the source trajectory transversely. 

To make sure the denominator in \eqref{dsdt} is bounded away from zero and the local solutions $s_j$'s are smooth, we consider the following construction. Since the components of $C$ are smooth, a function $s_j$ may fail to be smooth when the denominator in \eqref{dsdt} equals zero or when $s_j$ coincides with an endpoint of a segment. Let $\e>0$ be sufficiently small. Define the set
\begin{equation}\label{M-def}
\begin{split}
M:=\{(x,\Theta,s)\in {\overline U}\times S^2\times \cup_k [a_k+\e,b_k-\e]:
\Theta\cdot \dot\ga(s,x)=0,\ |\Theta\cdot \pa_s\dot\ga(s,x)|\ge \e\}.
\end{split}
\end{equation}
Here and below the bar denotes closure. Pick any $(x_0,\Theta_0)\in {\overline U}\times S^2$. Clearly, there can be at most finitely many points $s_j$ such that $(x_0,\Theta_0,s_j)\in M$. By construction, there exists a sufficiently small open neighborhood $V\ni (x_0,\Theta_0)$ such that each $s_j$ is a smooth function of $(x,\Theta)$ for all $(x,\Theta)\in V$. A collection of such $V$'s, one per each $(x_0,\Theta_0)\in {\overline U}\times S^2$, covers ${\overline U}\times S^2$. Choose a finite subcover, and let $N_m(x,\Theta)$ be a partition of unity subordinate to this subcover. On the support of each $N_m$ we have finitely many smooth functions $s_{mj}(x,\Theta)$, $j=1,\dots,J_m$. To simplify the notation, in what follows we replace each of the $N_m$ in the partition of unity with its $J_m$ copies, replace each copy of $N_m$ with $N_m/J_m$, and on the support of the $j$-th copy consider only one solution $s_{mj}(x,\Theta)$. The resulting partition of unity will be denoted $\{N_j\}$, and the corresponding solutions (one per each $N_j$) will be denoted $s_j$. These are precisely the solutions that have been used earlier in this section.

Our construction ensures that the denominator in \eqref{dsdt} is bounded away from zero and each $s_j(x_0,\Theta)$ is smooth on $\ts N_j(x_0,\Theta)$. Clearly, there is at most finitely many such solutions.

Condition \eqref{denom} can be viewed from another perspective. Consider $\Theta$ rotating so that $s=s_j(x_0,\Theta)$ remains constant. Then $N_j(x_0,\Theta)$ needs to be zero in a neighborhood of the direction
\begin{equation}\label{crit-dirs}
\Theta_{crit}(s,x_0):=
\frac{\dot\ga(s,x_0)\times {\pa_s}\dot\ga(s,x_0)}{|\dot\ga(s,x_0)\times {\pa_s}\dot\ga(s,x_0)|}.
\end{equation}
 
   Multiply \eqref{der-int} by $N_j$ and integrate over $S^2$ with respect to $\Theta$. This gives
\begin{equation}\label{almost-inv}
\begin{split}
\int_{S^2} & N_j(x_0,\Theta)\left.\frac{d}{d t}Q_j(x_t,\Theta)\right|_{t=0}d\Theta\\
=&\frac1{2\pi}\int_{S^2}\int_{\br}\int_{\br^3} f(x) N_j(x_0,\Theta) \frac{\pa}{\pa s}\left(\frac{A(s,x)}{A(s,x_0)}\right)\left.\frac{d s_t}{d t}\right|_{t=0} e^{i\Psi_j(x,x_0,\la\Theta)}dx i\la d\la d\Theta\\
&-\frac1{2\pi}\int_{S^2}\int_{\br}\int_{\br^3} f(x)  N_j(x_0,\Theta) \frac{A(s,x)}{A(s,x_0)}\left.\frac{d}{d t}\left[\al_t\cdot(\psi(s_t,x)-z(s_t))\right]\right|_{t=0}\\
&\qquad\qquad\qquad \times e^{i\Psi_j(x,x_0,\la\Theta)}dx \la^2 d\la  d\Theta,\\
&\Psi_j(x,x_0,\la\Theta):=\la\al\cdot(\psi(s,x)-\psi(s,x_0)),\ \al=\al_j(x_0,\Theta),\ s=s_j(x_0,\Theta).
\end{split}
\end{equation}
Here we represented the $\delta$-function in terms of its Fourier transform. It can be seen that both integrals on the right in \eqref{almost-inv} can be expressed in terms of the variable $\xi=\la\Theta$. First, as is easily seen, $\int_{S^2}\int_{-\infty}^0 (\cdot) d\la d\Theta= \int_{S^2}\ioi (\cdot) d\la d\Theta$ in both integrals. With $\la>0$, we have $\Theta=\xi/|\xi|$ and $s=s_j(x_0,\xi/|\xi|)$. Also, the term $\left.\frac{d}{d t}\left[\al_t\cdot(\psi(s_t,x)-z(s_t))\right]\right|_{t=0}$ is even in $\Theta$. Indeed, the expression in brackets is odd in $\Theta$. The derivative $\left.\frac{d}{d t}[\cdot]\right|_{t=0}$ is essentially the directional derivative along $\Theta$, which makes the result even. From \eqref{al-fam},
\begin{equation}\label{laal}
\la\al=\la\al_j(x_0,\Theta)=d_x\psi(s_j(x_0,\xi),x_0)^{-T}\xi.
\end{equation}
Here and in what follows, all functions of $\Theta$ (e.g., $s_j$, $\al_j$, $N_j$) are extended to $\br^3\setminus\{0\}$ as homogeneous of degree zero.

In order to integrate with respect to $\xi$ in the first integral on the right in \eqref{almost-inv} we need an extra factor $\la$. Clearly,
\begin{equation}\label{extra-la}
\left.\frac{d s_t}{d t}\right|_{t=0}=\left(\la \left.\frac{d s_t}{d t}\right|_{t=0}\right)\frac1{|\xi|^2}\la.
\end{equation}
Using \eqref{dsdt} we see that the factor in parentheses in \eqref{extra-la} is a function homogeneous of degree one in $\xi$, and the desired ``extra" $\la$ is found. Let 
\begin{equation}\label{amplitude}
\begin{split}
B_j(x,x_0,\xi):=&\left( \pa_s \frac{A(s,x)}{A(s,x_0)}\right)\left(\la \left.\frac{d s_t}{d t}\right|_{t=0}\right)\frac{i}{|\xi|^2}\\
&-\frac{A(s,x)}{A(s,x_0)}\left.\frac{d}{d t}\left[\al_t\cdot(\psi(s_t,x)-z(s_t))\right]\right|_{t=0},\\
\al=&\al_j(x_0,\xi),\ s=s_j(x_0,\xi).
\end{split}
\end{equation}
By construction, $N_j(x_0,\xi)B_j(x,x_0,\xi)\in\ci(U\times U\times (\br^3\setminus \{0\}))$. Moreover,
\begin{equation}\label{ampl-prop}
B_j(x,x_0,\xi)=1+O(|x-x_0|)+O(1/|\xi|),\ (x,x_0,\xi)\in U\times \ts N_j.
\end{equation}
The term $O(|x-x_0|)$ is uniform in $\xi$, and the term $O(1/|\xi|)$ is uniform in $x,x_0$. From \eqref{theta}, \eqref{almost-inv}, \eqref{laal}, 
\begin{equation}\label{phase-prop}
\Psi_j(x,x_0,\xi)=\xi\cdot (x-x_0)+O(|\xi||x-x_0|^2),
\end{equation}
and the big-$O$ term is smooth on $U\times \ts N_j$. Next we consider the zero-set $C_{\Psi_j}$:
\begin{equation}\label{canon-rel}
C_{\Psi_j}:=\{(x,x_0,\xi)\in U\times \ts N_j:\, d_\xi\Psi_j(x,x_0,\xi)=0\}.
\end{equation}
Obviously,
\begin{equation}\label{canon-rel-1}
\Delta_j:=\{(x,x,\xi)\in U\times \ts N_j\}\subset C_{\Psi_j}.
\end{equation}
We need to make sure that no other points belong to $C_{\Psi_j}$. For general deformations this property may not hold, so an additional restriction is needed. In this paper we make an additional assumption which guarantees that $\Delta_j=C_{\Psi_j}$. 

Let us look at the condition $\Delta_j=C_{\Psi_j}$ in more detail. It is convenient to represent $\Psi_j$ in the form
(cf. \eqref{almost-inv}, \eqref{laal}):
\begin{equation}\label{psi-alt}
\begin{split}
&\Psi_j(x,x_0,\xi)=\eta\cdot (y-y_0)=\eta\cdot (y-z(s_j)),\\
&\eta:=d_x\psi(s_j,x_0)^{-T}\xi,\ y:=\psi(s_j,x),\ y_0:=\psi(s_j,x_0).
\end{split}
\end{equation}
Recall that with the above notations we have (cf. \eqref{s-eqn})
\begin{equation}\label{sj-def}
\eta\cdot(y_0-z(s_j))\equiv0.
\end{equation}
Condition $d_\xi \Psi_j=0$ means that the first order partial derivatives of $\Psi_j$ with respect to $\xi$ vanish. Differentiating \eqref{psi-alt} along the direction of $\xi$ implies $\Psi_j(x,x_0,\xi)=0$, i.e. $\eta\cdot (y-y_0)=0$. Differentiating \eqref{psi-alt} along the direction that makes $\eta$ rotate in the plane $(y_0-z(s_j))^\perp$ so that $s_j$ does not change (cf. \eqref{sj-def}), proves that $y-z(s_j)$ and $y_0-z(s_j)$ are parallel, i.e.
\begin{equation}\label{two-betas}
\bt(s_j,x_0)=\bt(s_j,x).
\end{equation}
Finally, we differentiate $\Psi_j$ in \eqref{psi-alt} along the direction that makes $\eta$ rotate along $y_0-z(s_j)$, i.e. $\eta'=\kappa(y_0-z(s_j))/|y_0-z(s_j)|$ for some $\kappa \not= 0$. Let $s_j'$ be the corresponding derivative of $s_j$. This gives
\begin{equation}\label{dir-der-psi}
\kappa |y-z(s_j)|+\eta\cdot \pa_s (y-z(s_j))s_j'=0.
\end{equation}
Differentiating \eqref{sj-def} along the same direction we obtain
\begin{equation}\label{dir-der-psi-2}
\kappa |y_0-z(s_j)|+\eta\cdot \pa_s(y_0-z(s_j))s_j'=0.
\end{equation}
From \eqref{dir-der-psi-2} it follows that $s_j'\not=0$. Combining \eqref{dir-der-psi} and \eqref{dir-der-psi-2} gives
\begin{equation}\label{alt-cond-3}
\frac{\eta\cdot\pa_s(y-z(s_j))}{|y-z(s_j)|}
=\frac{\eta\cdot \pa_s(y_0-z(s_j))}{|y_0-z(s_j)|}.
\end{equation}
Thus, a simple manipulation shows that $\Delta_j=C_{\Psi_j}$ is equivalent to:
\begin{equation}\label{cond-3-final}
\bt(s_j,x)=\bt(s_j,x_0) \text{ and } \xi\cdot d_x\psi(s_j,x_0)^{-1}\pa_s(\bt(s_j,x)-\bt(s_j,x_0))=0
\implies x=x_0.
\end{equation}

Similarly to \eqref{crit-dirs}, condition \eqref{cond-3-final} can be viewed from another perspective. Given $x_0\in U$ and $s=s_j(x_0,\xi)$, the function $N_j(x_0,\xi)$ should be zero in a neighborhood of the set of directions
\begin{equation}\label{crit-arc}
\begin{split}
\Xi_{crit}(s,x_0):=\{\xi=
\la d_x\psi(s,x_0)^{-T}\left[\bt(s,x_0)\times \pa_s(\bt(s,x)-\bt(s,x_0))\right],\\
\la\not=0,x\in U,\bt(s,x)=\bt(s,x_0)\}.
\end{split}
\end{equation}
If $\bt(s,x_0)\times \pa_s(\bt(s,x)-\bt(s,x_0))$ is the zero vector for some $x$ on the curve $\bt(s,x)=\bt(s,x_0)$, then $\Xi_{crit}(s,x_0):=S^2$. 

If opposite directions are identified, then, generally, the intersection of $\Xi_{crit}$ with the unit sphere consists of an arc in the plane $\dot\ga(s,x_0)^\perp$. The arc is parametrized by the point $x$ moving along the curve $\bt(s,x)=\bt(s,x_0)$. It is interesting to note that $\Theta_{crit}(s,x_0)$ (cf. \eqref{crit-dirs}) is one of the endpoints of the arc. This endpoint corresponds to the case when $x\to x_0$ (see Appendix~\ref{sec:app}). Of course, in the case of integrals along lines (i.e., when $\psi(s,x)\equiv x$ for all $s\in I$) the set $\Xi_{crit}\cap S^2$ consists of a single point $\Theta_{crit}$.

We make two observations based on this fact. First, if $\psi$ is not the identity function, then, generally, additional artifacts can appear because of the exceptional directions in \eqref{crit-arc}. The second one is that if the deformation is sufficiently small (i.e., the functions $\psi(s,x)$, $s\in I$, are close to the identity map and change with $s$ sufficiently slowly), then each arc is sufficiently close to a single point, so the entire arcs will be cut-off by the partition of unity $\{N_j\}$. For general transformations, when the arcs are not too short, we make the assumption that a partition of unity $\{N_j\}$ can be found to cut off the critical directions. This means, in particular, that $\Xi_{crit}(s,x_0)\not=S^2$ for any $(s,x_0)\in I\times U$.

Note that phase functions somewhat similar to $\Psi_j$ are known in the literature, see e.g. \cite{belk}. However, even though the phase functions in \cite{belk} and in this paper look similar, there is an important distinction between them. The one in \cite{belk} is symmetric with respect to space variables. Our phase function is not symmetric: both $s$ and $\al$, which appear in $\Psi_j$, depend on $x_0$, but not on $x$.

Summing \eqref{almost-inv} over all $j$ and dividing by $8\pi^2$ gives
\begin{equation}\label{pdo-inv}
\begin{split}
(\CB f)(x_0):=&\frac1{8\pi^2}\int_{S^2} \sum_j N_j(x_0,\Theta)\left.\frac{d}{d t}Q_j(x_t,\Theta)\right|_{t=0}d\Theta\\
=&\frac1{(2\pi)^3}\int_{\br^3}\int_{\br^3} f(x) \sum_j N_j(x_0,\xi) B_j(x,x_0,\xi) e^{i\Psi_j(x,x_0,\xi)}dx d\xi.
\end{split}
\end{equation}
Recall that, by construction, the sum in \eqref{pdo-inv} is finite for any $(x_0,\Theta)\in U\times S^2$. 

Let us summarize what we have so far. 
\begin{enumerate}
\item By assumption, $C_{\Psi_j}=\Delta_j$ for all $j$. \label{ass1}
\item Given any $(x_0,\xi)\in U\times (\br^3\setminus\{0\})$, there is $j$ such that $N_j(x_0,\xi)>0$. 
\item The amplitude $B_j$ and phase $\Psi_j$ satisfy
\begin{equation}\label{amphsum}
\begin{split}
B_j(x,x_0,\xi)&=1+O(|x-x_0|)+O(1/|\xi|),\\ 
\Psi_j(x,x_0,\xi)&=\xi\cdot (x-x_0)+O(|\xi||x-x_0|^2),\ (x,x_0,\xi)\in U\times\ts N_j.
\end{split}
\end{equation} 
\end{enumerate}
From \eqref{amphsum} and assumption (\ref{ass1}) above it is clear that $\Psi_j$ is a nondenerate phase function (cf. \cite{dus}, p. 31). Finally, it follows immediately from the construction of $\Psi_j$ that $\pa_x \Psi_j(x,x_0,\xi)=-\pa_{x_0} \Psi_j(x,x_0,\xi)$ when $x=x_0$ and $(x_0,\xi)\in\ts N_j$. Hence $\CB$ is a PDO (cf. \cite{dus}, p. 45). 

Properties \eqref{amphsum} prove that $\CB$ is an elliptic PDO of order zero with principal symbol 1 (recall that $\sum N_j\equiv 1$). On the other hand, the left side of \eqref{pdo-inv} is computable from the cone beam data. Thus, we obtained the desired approximate inversion formula. Since $\CB$ is an elliptic PDO, all the singularities are preserved, and added ones do not appear. We formulate our result as a theorem.

\begin{theorem}\label{thm1} 
Suppose there exist a finite partition of unity $\{N_j\}$ on ${\overline U}\times S^2$ and the corresponding smooth solutions $s_j$ to the equation
\begin{equation}\label{s-eqn-2}
\Theta\cdot \dot\ga(s,x_0)=0,\ s=s_j(x_0,\Theta), (x_0,\Theta)\in\ts N_j,
\end{equation}
with the following properties. If $(x_0,\Theta)\in\ts N_j$ and $s=s_j(x_0,\Theta)$, then
\begin{enumerate}
\item $\Theta\not\in \overline{\Xi_{crit}(s,x_0)}$, and
\item $z(s)$ is not an endpoint of $C$. 
\end{enumerate}
Denote
\begin{equation}\label{q-def-2}
Q_j(x_0,\Theta):=\frac1{A(s,x_0)} \int_{S^2} X_{f_s}(z(s),\bt)\delta'(\al\cdot\bt)d\bt,\ 
\al=\al_j(x_0,\Theta),s=s_j(x_0,\Theta).
\end{equation}
Let $x_t,\al_t$, and $s_t$ be as defined in \eqref{t_vars}. Then the operator 
\begin{equation}\label{pdo-inv-2}
\begin{split}
(\CB f)(x_0):=\frac1{8\pi^2}\int_{S^2} \sum_j N_j(x_0,\Theta)\left.\frac{d}{d t}Q_j(x_t,\Theta)\right|_{t=0}d\Theta
\end{split}
\end{equation}
is an elliptic PDO of order zero with principal symbol 1.
\end{theorem}

From the equations \eqref{almost-inv} and \eqref{pdo-inv} we see that by using the intermediate function $Q$, which is based on an integral of the ray tranform, we get access to the frequency variable $\Theta$. This allows us to incorporate the cut-offs $N_j$ at the backprojection step and eliminate undesirable directions $\Xi_{crit}$, which otherwise would have lead to added singularities. 



\section{Analysis of the inversion formula} \label{sec:anal_inv}

In this section we suppose that the deformation is close to the identity. To be precise, we assume that $\psi$ and $A$ depend on a parameter $\e$, and the following assumptions hold:
\begin{equation}\label{smallness}
A_\e(s,x)\to 1 \text{ and } \psi_\e(s,x) \to x \text{ in } \ci(I\times \br^3) \text{ as }\e\to0.
\end{equation}
Thus, for any index $k$ and any multiindex $m$ we have:
\begin{equation}\label{seminorms}
\sup_{(s,x)\in I\times \br^3}|\pa_s^k\pa_x^m(A_\e(s,x)-1)|\to0,\
\sup_{(s,x)\in I\times \br^3}|\pa_s^k\pa_x^m(\psi_\e(s,x)-x)|\to0 \text{ as }\e\to0.
\end{equation}

\begin{theorem}\label{thm2} 
Suppose \eqref{smallness} holds. Pick any $\nu\in\br$ and consider the operators $\CB_\e-\text{Id}: H^\nu_0(U)\to H^{\nu-1}_{loc}(U)$. Then $\CB_\e-\text{Id}\to0$ as $\e\to0$.
\end{theorem}

\begin{proof} 
We must prove that for every $g\in\coi(U)$ and for every compact $K\subset U$, there is a constant $C_\e>0$ such that
\begin{equation}\label{cont}
\Vert g(\CB_\e-\text{Id})f\Vert_{\nu-1}\le C_\e \Vert f \Vert_\nu,\ \forall f\in\coi(K),
\end{equation}
and $C_\e\to0$ as $\e\to0$.

For simplicity, in what follows the dependence of various functions on $\e$ is omitted from  notations. 

Since $\sum_j N_j(x_0,\Theta)\equiv1$ on $U\times S^2$, \eqref{amplitude} implies that we must show that the 
PDOs 
\begin{equation}\label{new-PDOs}
(\CB_{kj}f)(x_0):=\frac1{(2\pi)^3} \int_{\br^3}\int_{\br^3} B_{kj}(x,x_0,\xi) f(x) e^{i\Psi_j(x,x_0,\xi)}dxd\xi,\ k=1,2,
\end{equation}
where 
\begin{equation}\label{new-amplitude}
\begin{split}
B_{1j}(x,x_0,\xi):=&g(x_0)h(x)N_j(x_0,\xi)\left( \frac{\pa}{\pa s} \frac{A(s,x)}{A(s,x_0)}\right)\left(\la \left.\frac{d s_t}{d t}\right|_{t=0}\right)\frac{i}{|\xi|^2},\\
B_{2j}(x,x_0,\xi):=&g(x_0)h(x)N_j(x_0,\xi)\biggl(\frac{A(s,x)}{A(s,x_0)}\left.\frac{d}{d t}\left[\al_t\cdot(\psi(s_t,x)-z(s_t))\right]\right|_{t=0}+1\biggr),\\
\al=&\al_j(x_0,\xi),\ s=s_j(x_0,\xi),
\end{split}
\end{equation}
converge in norm to the zero operator $H^\nu_0(U)\to H^{\nu-1}_0(U)$ for any $j$. Here we inserted $h\in \coi(U)$ such that $h\equiv 1$ on $K$.

Pick any function $\chi\in\ci(\br^3)$ satisfying $\chi(\xi)=0$ if $|\xi|\leq1$ and $\chi(\xi)=1$ if $|\xi|\ge2$. We will analyze the operators with amplitudes $B_{kj}(x,x_0,\xi)(1-\chi(\xi))$ and $B_{kj}(x,x_0,\xi)\chi(\xi)$. We start by looking at the latter. Until mentioned otherwise, our standing assumption in what follows is
\begin{equation}\label{st-ass}
(x,x_0,\xi)\in U\times \ts N_j,\ |\xi|\ge1.
\end{equation}

Define the function $\eta=\eta(x,x_0,\xi)$ from the equation
\begin{equation}\label{tilde-th}
\xi\cdot d_x\psi(s_j,x_0)^{-1}(\psi(s_j,x)-\psi(s_j,x_0))
=\eta \cdot (x-x_0).
\end{equation}
Using the Taylor expansion and \eqref{seminorms}, rewrite \eqref{tilde-th} in the form:
\begin{equation}\label{tilde-th-2}
\xi\cdot \left[\text{Id}+o_\e(1)(x-x_0)\right](x-x_0)
=\eta \cdot (x-x_0).
\end{equation}
Here $o_\e(1)$ is a tensor of order three, which is homogeneous of degree zero in $\xi$. The subscript $\e$ in $o_\e(1)$ means that the latter becomes small as $\e\to0$. In what follows, the same notation $o_\e(1)$ is used for various kinds of function (e.g., matrix-valued, vector-valued, etc.). From the context it will be clear what kind of function is assumed in each particular case. Because of \eqref{seminorms}, $o_\e(1)$ goes to zero with all derivatives as $\e\to0$ uniformly with respect to $(x,x_0,\xi)$ (cf. \eqref{st-ass}). Therefore, we can define
\begin{equation}\label{tth-def}
\eta(x,x_0,\xi):=\left[\text{Id}+o_\e(1)(x-x_0)\right]^T\xi.
\end{equation}
If $\e$ is small enough, then \eqref{tth-def} can be solved for $\xi$ in terms of $\eta$ and 
\begin{equation}\label{det-to-1}
\text{det}(\pa\xi/\pa\eta)= 1+o_\e(1)(x-x_0)\text{ as } \e\to0.
\end{equation}
As before, $o_\e(1)$ goes to zero with all derivatives uniformly with respect to $(x,x_0,\xi)$ in the indicated set (cf. \eqref{st-ass}). It is important to point out that the assumption \eqref{st-ass} has different meanings  in \eqref{tth-def} and \eqref{det-to-1}. In \eqref{tth-def}, $\xi$ is an independent variable. In \eqref{det-to-1}, 
$\xi$ is a function of $x$, $x_0$, and $\eta$. Thus, in \eqref{det-to-1}, the assumption \eqref{st-ass} means that $x\in U$ and $(x_0,\xi(x,x_0,\eta))\in\ts N_j$. This meaning will be implied in what follows whenever $\xi$ is a dependent variable.

Changing variables we obtain PDOs of the type
\begin{equation}\label{new-PDOs-2}
\frac1{(2\pi)^3} \int_{\br^3}\int_{\br^3} B_{kj}(x,x_0,\xi)\chi(\xi)\text{det}\left(\frac{\pa\xi}{\pa\eta}\right) f(x) e^{i\eta\cdot(x-x_0)}dxd\eta,
\end{equation}
where $\xi=\xi(x,x_0,\eta)$. Consider first $B_{1j}$. From \eqref{new-amplitude} and \eqref{new-PDOs-2}, after multiplying $B_{1j}$ by $|\xi|$ we obtain 
\begin{equation}\label{need-to-show}
\begin{split}
g(x_0)h(x)&N_j(x_0,\xi)\chi(\xi)\left( \pa_s \frac{A(s,x)}{A(s,x_0)}\right) \\
&\times\left[\frac1{|\xi|}\left.\frac{d s_t}{d t}\right|_{t=0}\right]\text{det}\left(\frac{\pa\xi}{\pa\eta}\right)\to0 \text{ in }\ci(U\times U\times  \br^3),\ \e\to0.
\end{split}
\end{equation}
Indeed, utilizing \eqref{smallness}, \eqref{det-to-1}, and observing that $\left.\frac{d s_t}{d t}\right|_{t=0}$ is bounded with all derivatives, \eqref{need-to-show} immediately follows. Note that the expression in \eqref{need-to-show} is homogeneous of degree zero in $\eta$ (for large $|\eta|)$.

To analyze $B_{2j}$ we need an intermediate result. Obviously,
\begin{equation}\label{A-est}
\frac{A(s,x)}{A(s,x_0)}=1+o_\e(1)(x-x_0).
\end{equation}
Also, similarly to \eqref{der-c1} and \eqref{tilde-th-2}, we obtain
\begin{equation}\label{der-c1-alt}
\begin{split}
&\left.\frac{d}{d t}\left[\al_t\cdot(\psi(s_t,x)-z(s_t))\right]\right|_{t=0} \\
&=\left.\frac{d}{d t}\left[\al_t\cdot(\psi(s_t,x)-\psi(s_t,x_t))\right]\right|_{t=0}\\
&=\left.\frac{d}{d t}\left[\al_t\cdot\left\{d_x\psi(s_t,x_t)(x-x_t)+o_\e(1)(x-x_t,x-x_t)\right\}\right]\right|_{t=0}\\
&=\left.\frac{d}{d t}\left[\Theta\cdot(x-x_t)+o_\e(1)(x-x_t,x-x_t)\right]\right|_{t=0}\\
&=-1+o_\e(1)(x-x_0).
\end{split}
\end{equation}
In the third and fourth lines of \eqref{der-c1-alt}, the two copies of $x-x_t$ are input vectors on which the degree-three tensor $o_\e(1)$ operates. Combining \eqref{det-to-1}, \eqref{A-est}, and \eqref{der-c1-alt}, we get 
\begin{equation}\label{part2}
\frac{A(s,x)}{A(s,x_0)}\left.\frac{d}{d t}\left[\al_t\cdot(\psi(s_t,x)-z(s_t))\right]\right|_{t=0}\text{det}\left(\frac{\pa\xi}{\pa\eta}\right)+1=o_\e(1)(x-x_0).
\end{equation}
Consequently,
\begin{equation}\label{part2-2}
\begin{split}
g(x_0)h(x)&N_j(x_0,\xi)\chi(\xi)\\
&\times\biggl[\frac{A(s,x)}{A(s,x_0)}\left.\frac{d}{d t}\left[\al_t\cdot(\psi(s_t,x)-z(s_t))\right]\right|_{t=0}\text{det}\left(\frac{\pa\xi}{\pa\eta}\right)+1\biggr]\\
=&o_\e(1)(x-x_0)\to 0 \text{ in } \ci(U\times U\times \br^3).
\end{split}
\end{equation}
The PDO with the amplitude \eqref{part2-2} is given by
\begin{equation}\label{pdo-part2}
\int_{\br^3}\int_{\br^3} f(x) [\chi(\xi)o_\e(1)(x-x_0)] e^{i\eta\cdot (x-x_0)}dx d\eta.
\end{equation}
Recall that $o_\e(1)$ in \eqref{pdo-part2} is homogeneous of degree zero in $\eta$. Integrating by parts with respect to $\eta$ in \eqref{pdo-part2} (see e.g. \cite{trev1}, p. 33) results in terms of the type: $\chi'(\xi)o_\e(1)$ and $\chi(\xi)o_\e(1)$ (recall that $\xi$ is a function of $\eta$).
In the first one, $\chi'(\xi)$ is compactly supported and $o_\e(1)$ is homogeneous of degree zero in $\eta$.
In the second one, $o_\e(1)$ is homogeneous of degree -1 in $\eta$. In both cases, the factors $o_\e(1)$  remain stable when differentiated with respect to $x$, $x_0$, and $\eta$. Combining with \eqref{need-to-show} we prove that every $S^{-1}_{1,0}$ seminorm of the amplitude of the PDOs in \eqref{new-PDOs-2} goes to zero as $\e\to0$. Using the conventional argument (see e.g. \cite{trev1}, pp. 17, 18), we prove that the PDOs in \eqref{new-PDOs-2} go to zero in norm as operators $H^\nu_0(U)\to H^{\nu-1}_0(U)$.

To finish the proof we need to look at PDOs of the type:
\begin{equation}\label{new-PDOs-3}
\frac1{(2\pi)^3} \int_{\br^3}\int_{\br^3} B_{kj}(x,x_0,\xi)(1-\chi(\xi)) f(x) e^{i\Psi_j(x,x_0,\xi)}dxd\xi.
\end{equation}
The change of variables $\xi\to\eta$ is not needed here, and the assumption \eqref{st-ass} does not apply. We have: (i) $B_{kj}(\cdot,\cdot,\xi)\in \coi(U\times U)$ and $\Psi_j(\cdot,\cdot,\xi)\in \ci(U\times U)$, (ii) $|\xi|B_{1j}(x,x_0,\xi)=o_\e(1)$, $B_{2j}(x,x_0,\xi)=o_\e(1)$, and both $o_\e(1)$ terms are stable when differentiated with respect to $x$, $x_0$, and (iii) $B_{kj}$, $k=1,2$, are integrable at the origin $\xi=0$. Hence the desired assertion follows.
\end{proof}

\section{Localized inversion in the static case} \label{sec:alt_inv}

To illustrate the idea of localized reconstruction we consider an important static case. The data are
\begin{equation}\label{data-static}
X_f(s,\bt):=\ioi f(z(s)+t\bt)dt,\ s\in I,\bt\in S^2.
\end{equation}
Consider the integral arising in the proof of the Grangeat formula. For simplicity we assume first that the plane of integration is perpendicular to the $x_3$-axis, and the source is located at the point $z$. Let $\phi\in\ci(\br^3\setminus\{0\})$ be a function homogeneous of degree zero. Denoting
\begin{equation}\label{Theta-eps}
u_\e(\theta):=(\sqrt{1-\e^2}\cos\theta,\sqrt{1-\e^2}\sin\theta,\e)\in S^2,
\end{equation}
we have
\begin{equation}\label{gr-start}
\begin{split}
\int_0^{2\pi}&\left.\frac{d}{d\e}\ioi f(z+t u_\e(\theta)) \phi(t u_\e(\theta)) dt \right|_{\e=0}d\theta\\
&=\int_{\br^2} \left.\frac{\pa}{\pa x_3} f(z_1+x_1,z_2+x_2,z_3+x_3)\phi(x_1,x_2,x_3)\right|_{x_3=0} dx_1 dx_2\\
&=-\int_{\br^3} f(z+x)\phi(x)\delta'(e_3\cdot x)dx,
\end{split}
\end{equation}
where $e_3$ is the unit vector along the $x_3$-axis. Since $\phi$ is homogeneous of degree zero, the left side of \eqref{gr-start} can be computed from the data \eqref{data-static}. In coordinate-free form equation \eqref{gr-start} can be written similarly to \eqref{grang}, \eqref{rtder}:
\begin{equation}\label{grang-static}
\int_{S^2} X_f(z,\bt)\phi(\bt)\delta'(\al\cdot\bt)d\bt
=\int_{\br^3} f(x)\phi(x-z)\delta'(\al\cdot(x-z))dx.
\end{equation}
In \eqref{grang-static} the plane of integration and the reconstruction point are assumed to be fixed. Thus, the function $\phi$ may also depend on $\al$ and $x_0$. Assuming the source trajectory satisfies the Kirillov-Tuy condition, for each $(x_0,\al)\in U\times S^2$ we can find locally smooth solutions $s=s_j(x_0,\al)$ to the equation
\begin{equation}\label{ips}
\al\cdot(x_0-z(s))=0.
\end{equation}
Substituting $z=z(s_j)$ into \eqref{grang-static} gives
\begin{equation}\label{grang-static-2}
\begin{split}
Q_j(x_0,\al):&=\int_{S^2} X_f(z(s_j),\bt)\phi(\bt;x_0,\al)\delta'(\al\cdot\bt)d\bt\\
&=\int_{\br^3} f(x)\phi(x-z(s_j);x_0,\al)\delta'(\al\cdot(x-x_0))dx,\ s_j=s_j(x_0,\al).
\end{split}
\end{equation}
By construction, $Q_j(x_0,\al)$ can be computed from the data. Similarly to Section~\ref{sec:first_inv}, define $x_t=x_0+t\al$.  Substituting into \eqref{grang-static-2} and differentiating gives
\begin{equation}\label{grang-deriv}
\begin{split}
\left.\frac{d}{dt} Q_j(x_t,\al)\right|_{t=0}=&-\int_{\br^3} f(x)\phi(x-z(s_j);x_0,\al)\delta''(\al\cdot(x-x_0))dx\\
&+\int_{\br^3} f(x)\left.\frac{d}{dt} \phi(x-z(s_j(x_t,\al));x_t,\al)\right|_{t=0}\delta'(\al\cdot(x-x_0))dx.
\end{split}
\end{equation}
Clearly,
\begin{equation}\label{phi-deriv}
\begin{split}
\left.\frac{d}{dt} \phi(x-z(s_j(x_t,\al));x_t,\al)\right|_{t=0}
=d_y\phi(x-z(s_j(y,\al));y,\al)|_{y=x_0}\al.
\end{split}
\end{equation}

Let $\{N_j(x,\al)\}$ be a smooth, finite partition of unity on ${\overline U}\times S^2$ constructed as in Section~\ref{sec:first_inv}. Multiplying \eqref{grang-deriv} by $N_j(x_0,\al)$, summing over all $j$, dividing by $8\pi^2$, and arguing similarly to \eqref{almost-inv}--\eqref{extra-la}, we obtain the analogues of \eqref{pdo-inv} and \eqref{amplitude}:
\begin{equation}\label{pdo-inv-static}
\begin{split}
(\CB f)(x_0):=&\frac1{8\pi^2}\int_{S^2} \sum_j N_j(x_0,\al)\left.\frac{d}{d t}Q_j(x_t, \al)\right|_{t=0}d\al\\
=&\frac1{(2\pi)^3}\int_{\br^3}\int_{\br^3} f(x) \sum_j N_j(x_0,\xi) B_j(x,x_0,\xi) e^{i\xi\cdot(x-x_0)}dx d\xi,
\end{split}
\end{equation}
where
\begin{equation}\label{amplitude-pdo-static}
\begin{split}
B_j(x,x_0,\xi)=\phi(x-z(s_j(x_0,\xi));x_0,\xi)+
i \frac{d_y\phi(x-z(s_j(y,\xi));y,\xi)|_{y=x_0}\xi}{|\xi|^2}.
\end{split}
\end{equation}
In \eqref{amplitude-pdo-static}, $\phi$ and $s_j$, as functions of $\al$, are extended from $S^2$ to $\br^3\setminus\{0\}$ as homogeneous of degree zero.

Strictly speaking, $B_j$ is not an amplitude since $\phi$ in \eqref{amplitude-pdo-static} is not smooth in $x$ when $x=z(s)$. However, we can multiply $B_j$ by the cut-off $h(x)$ (cf. \eqref{new-amplitude}). This does not alter the operator $\CB$ acting on functions $f\in\coi(K)$, and the product $h(x)B_j(x,x_0,\xi)$ is an amplitude.

In order to have accurate reconstruction, we choose $\phi$ such that 
\begin{equation}\label{phi-prop-1}
\phi(x-z(s_j(y,\xi));y,\xi)\equiv 1,\ |x-y|<\e_1,\ x,y\in U,
\end{equation}
for some $\e_1>0$. Using \eqref{phi-prop-1} in \eqref{amplitude-pdo-static} implies 
\begin{equation}\label{derives-ampl}
B_j(x_0,x_0,\xi)\equiv 1,\ \pa_x^m B_j(x,x_0,\xi)|_{x=x_0}\equiv 0,\ |m|\ge 1, (x_0,\xi)\in\ts N_j,
\end{equation}
where $m$ is a multiindex. Hence, if \eqref{phi-prop-1} holds, the {\it symbol} of the PDO $\CB$ equals 1 (see \cite{dus}, Theorem 2.5.1).

In order to achieve localized reconstruction, we choose $\phi$ such that 
\begin{equation}\label{phi-prop-2}
\phi(x-z(s_j);y,\xi)\equiv 0\text{ if } \frac{x-z(s_j)}{|x-z(s_j)|}\cdot \frac{y-z(s_j)}{|y-z(s_j)|}<1-\e_2, \ s_j=s_j(y,\xi),
\end{equation}
for some $\e_2>0$. Obviously, given any $\e_2>0$ one can find $\e_1>0$ such that the conditions  \eqref{phi-prop-1} and \eqref{phi-prop-2} are non-contradictory. Thus, the inversion formula \eqref{pdo-inv-static} has two desirable properties: (i) it reconstructs $f$ up to a $\ci$ function, and (ii) given any $\e>0$, we can find the function $\phi$ such that reconstruction at any $x\in U$ uses integrals of $f$ along lines passing through an $\e$-neighborhood of $x$.

\section{Discussion and some generalizations} \label{sec:general}

Let us compare our results with a more traditional approach based on using $X^*X$. Here $X^*$ is a backprojection operator, which is related to the formal dual of $X$ and includes all the necessary cut-offs to make sure the composition $X^*X$ is well-defined. There is no need to insert any operator between $X^*$ and $X$, because we are interested in the location of added singularities, and not in their strength. For the same reason we ignore the weights in $X$ and $X^*$. Thus, we have
\begin{equation}\label{xstarx}
\begin{split}
(X^*Xf)(x_0)&=\int_I\ioi \chi_1(s)\chi_2(t) f(\nu(s,z(s)+t(\psi(s,x_0)-z(s))))dtds\\
&=\frac1{(2\pi)^3}\int_{\br^3}\int_{\br}\int_{\br} \int_{\br^3}f(x) \chi_1(s)\chi_2(t) e^{i\Psi(x,x_0;\eta,s,t)}dx dt ds d\eta,
\end{split}
\end{equation}
where
\begin{equation}\label{new-phase}
\Psi(x,x_0;\eta,s,t)=\eta\cdot (\nu(s,z(s)+t(\psi(s,x_0)-z(s)))-x).
\end{equation}
Here $\chi_1\in \coi(I)$ and $\chi_2\in\coi(\br_+)$. Changing variables $s=\tilde s/|\eta|$ and $t=\tilde t/|\eta|$ (cf. e.g. \cite{dus}, p. 40), we obtain that  $X^*X$ is a singular FIO (see \cite{gruhl}) with the frequency variables $\eta,\tilde s,\tilde t$, the amplitude $\chi_1(\tilde s/|\eta|)\chi_2(\tilde t/|\eta|)$, and the phase function $\tilde\Psi(x,x_0;\eta,\tilde s,\tilde t):=\Psi(x,x_0;\eta,\tilde s/|\eta|,\tilde t/|\eta|)$.  
As is easily seen, the condition $d_{\eta,\tilde s,\tilde t}\tilde \Psi=0$ is equivalent to the condition $d_{\eta,s,t}\Psi=0$. The latter gives
\begin{align}\label{cond1}
&y:=z(s)+t(\psi(s,x_0)-z(s))=\psi(s,x),\\ \label{cond2}
&\eta\cdot d_y\nu(s,y)\bt(s,x_0)=0,\\ \label{cond3}
&\eta\cdot (\pa_s\nu(s,y)+d_y\nu(s,y)\pa_s(z(s)+t(\psi(s,x_0)-z(s)))=0.
\end{align}

The operator $X^*X$ can add singularities because of two reasons: (i) the symbol of $X^*X$ is singular, and (ii) its canonical relation is not diagonal. First consider case (ii). Microlocally away from the singularity of the symbol, the canonical relation of $X^*X$ is diagonal if $d_{\eta,s,t}\Psi=0$ implies $x=x_0$. 
Condition \eqref{cond1} implies
\begin{equation}\label{bt-eq}
\bt(s,x)=\bt(s,x_0).
\end{equation}
By construction, $\nu(s,\psi(s,x))\equiv x$. Hence
\begin{equation}\label{easy-id}
\pa_s\nu(s,y)+d_y\nu(s,y)\pa_s\psi(s,x)\equiv0,\ y=\psi(s,x).
\end{equation}
Applying \eqref{easy-id} in \eqref{cond3} with $y$ defined in \eqref{cond1} and then using \eqref{cond2}, \eqref{bt-eq} we find
\begin{equation}\label{eq_a1}
\begin{split}
0&=\eta\cdot d_y\nu(s,y)[t\pa_s(\psi(s,x_0)-z(s))-\pa_s(\psi(s,x)-z(s))]\\
&=\eta\cdot d_y\nu(s,y)[tL_0\pa_s\bt(s,x_0)-L\pa_s\bt(s,x)],
\end{split}
\end{equation}
where
\begin{equation}\label{lengths}
L:=|\psi(s,x)-z(s)|,\ L_0:=|\psi(s,x_0)-z(s)|.
\end{equation}
From \eqref{cond1}, $L=tL_0$, so \eqref{eq_a1} implies
\begin{equation}\label{eq_a2}
\eta\cdot d_y\nu(s,y)\pa_s(\bt(s,x_0)-\bt(s,x))=0.
\end{equation}
Ignoring the inconsequential change of variables $\xi\leftrightarrow\eta$ according to
\begin{equation}\label{vars}
d_x\psi(s,x_0)^{-T}\xi \leftrightarrow d_x\psi(s,x)^{-T}\eta,
\end{equation}
conditions \eqref{bt-eq}, \eqref{cond2}, and \eqref{eq_a2} (or, \eqref{cond1}--\eqref{cond3}) are equivalent to conditions \eqref{s-eqn} (cf. \eqref{gadot-def}) and \eqref{cond-3-final}. 


Consider now case (i). As is seen from \eqref{dsdt} and \eqref{almost-inv}, the singularity of the symbol of $R^*\tilde D X$ occurs when $\Theta\cdot\pa_s\dot\ga(s,x_0)=0$. To find the top order symbol of $X^* X$ in a neighborhood of $(x=x_0,x_0,\eta)$, we need to compute the asymptotics of the integral $\int (\cdot) \exp(i\Psi(x_0,x_0;\eta=\sigma\Theta,s,t) dsdt$ as $\sigma\to\infty$. The critical point of the phase is $(t_0=1,s=s_0)$, where $s_0$ solves $\Theta\cdot\dot\ga(s,x_0)=0$. As before, an elementary calculation gives that the symbol is singular when $\Theta\cdot\pa_s\dot\ga(s,x_0)=0$.

The above argument shows that the mechanisms by which the operators $R^*\tilde D X$ and $X^* X$ can add singularities to the reconstructed image are essentially the same. Therefore, the key advantage of using $R^*\tilde D X$ compared with $X^* X$ is the ability to use cut-offs in the frequency domain and thereby eliminate both reasons leading to artifacts. That ability is based on the redundancies present in the restricted ray transform data. The redundancy is reflected in the existence of multiple solutions to the equation \eqref{s-eqn}. 

Note that not all source trajectories have enough redundancies to allow complete artifact removal even in the static case. For example, in the case of a helix there are planes that intersect the trajectory at only one point, and this intersection is tangential. On the other hand, another classical source trajectory - two orthogonal circles - does have enough redundancies to allow complete artifact removal in the static case. For general source trajectories $C$ and general deformations $\psi$, the condition that allows complete artifact removal can be stated as follows: {\it for any $(x_0,\Theta)\in U\times S^2$ there exists at least one non-critical solution $s$ to \eqref{s-eqn}}. Here ``non-critical'' is understood not in the narrow sense of \eqref{denom}, but in the more general sense of \eqref{crit-arc}.

Next we discuss various generalizations of the approaches proposed in the previous sections. Consider a collection of smooth curves parametrized by the arc length
\begin{equation}\label{curves}
\gamma_{s,q}(t),\ s\in I,q\in S^2,\ t\ge0, 
\end{equation}
$t$ is the parameter (the arc length) along the curves, $\gamma_{s,q}(0)=z(s)\in C$ (cf. \eqref{curve}), and $\dot\gamma_{s,q}(0)=q$ for any $s\in I$ and $q\in S^2$. It is convenient to think of $S^2$ as a two-dimensional detector. Our {\it main assumption} is that for each $s\in I$ the equation
\begin{equation}\label{x-to-curves}
x=\gamma_{s,q}(t),\ x\in \br^3\setminus \{z(s)\}, 
\end{equation}
has a unique smooth solution $t=\hat t(s,x)$, $q=\hat q(s,x)$, $\hat t,\hat q\in\ci(\{(s,x)\in I\times U:x\not=z(s)\})$. Continuing the medical analogy, $C$ is the x-ray source trajectory, and $\hat q(s,x)$ is the projection of the reconstruction point $x$ on the detector. To avoid confusion, $C$ will be called source trajectory, and $\gamma$'s will be called curves. The tomographic data are
\begin{equation}\label{gen-data}
X_f(s,q):=\ioi f(\gamma_{s,q}(t)) w_0(s,q,t)dt,\ s\in I,\ q\in S^2,
\end{equation}
for some smooth strictly positive weight $w_0$.


As is easily seen, there exists a family of smooth maps $y=\psi(s,x)$, $s\in I$, such that the images of the curves $\gamma_{s,q}(t)\to \psi(s,\gamma_{s,q}(t))$, $t>0$, are straight lines. For each $s\in I$, the map $\psi(s,\cdot)$ is given by
\begin{equation}\label{gp-psi}
\begin{split}
x \to \psi(s,x):=z(s)+\hat t(s,x)\hat q(s,x).
\end{split}
\end{equation}
By construction, $\psi(s,x)$ approaches the identity map as $x\to z(s)$. Thus, the algorithm described in Section~\ref{sec:first_inv} applies to a general class of ray transforms, and the assumption about the existence of ``deformation'' functions that map curves into lines is not restrictive. The assumption that these deformations become the identity map outside of some bounded set is not required either as long as $f$ is compactly supported.
The entire derivation in Section~\ref{sec:first_inv} can be made in terms of the original curves $\ga$ rather than in terms of their straightened out versions via the Grangeat formula. It may depend on a particular application whether the calculation in the original coordinates or the transformed coordinates is preferred.

Combining the idea of Section~\ref{sec:alt_inv} with the algorithm of Section~\ref{sec:first_inv} shows that by introducing a cut-off function $\phi$ the algorithm can be made to use only $\ga$'s passing through a small neighborhood of a reconstruction point $x_0$. In this case the result of reconstruction can still be written in the form $\CB f$, where $\CB$ is an elliptic PDO with principal symbol 1.

The algorithms of Sections~\ref{sec:first_inv} and \ref{sec:alt_inv} are based on integrating the derivative of the cone beam data to obtain an intermediate function $Q$ (Step 1) and then backprojecting the derivative of $Q$ (Step 2). See \eqref{grang}, \eqref{integral}, and \eqref{pdo-inv-2} in Section~\ref{sec:first_inv} as well as \eqref{grang-static}, \eqref{grang-static-2}, and \eqref{pdo-inv-static} in Section~\ref{sec:alt_inv}. In fact, the distribution of derivatives across the two steps is fairly flexible. For instance, one can use a second order derivative in Step 1 and no derivatives in Step 2, or -- no derivatives in Step 1 and a second order derivative in Step 2. In each of these cases one gets an elliptic PDO with principal symbol 1. Even more generally, if an $m$-th order derivative is used in Step 1, and an $n$-th order derivative is used in Step 2, then one gets an elliptic PDO of order $m+n-2$. The latter can then be inverted (modulo $\ci$) by its parametrix. In each of these cases the phase function does not change and remains equal to $\Psi_j$.

\appendix
\section{Finding an endpoint of an arc of critical directions.}\label{sec:app}

Throughout this section we assume $s=s_j(x_0,\Theta)$. To enforce the condition $\bt(s,x)=\bt(s,x_0)$ suppose that $x=x(\e)$  satisfies 
\begin{equation}\label{x-curve}
\psi(s,x)-\psi(s,x_0)=\e (\psi(s,x_0)-z(s))
\end{equation}
for $\e$ small. Thus, we also have (cf. \eqref{lengths})
\begin{equation}\label{L-rel}
L-L_0=\e L_0.
\end{equation}
To see what happens with \eqref{cond-3-final} as $x\to x_0$, we can consider the limit of
\begin{equation}\label{step-zero}
\frac1\e\left[\frac{\pa_s(\psi(s,x)-z(s))}L-\frac{\pa_s(\psi(s,x_0)-z(s))}{L_0}\right]
\end{equation}
as $\e\to0$. Here we have used that, in view of \eqref{sj-def}, there is no need to differentiate $1/L$ and $1/L_0$. In view of \eqref{L-rel}, the expression in \eqref{step-zero} transforms to
\begin{equation}\label{step-1}
\begin{split}
\frac1\e\left[\frac{\pa_s(\psi(s,x)-\psi(s,x_0))}L+\pa_s(\psi(s,x_0)-z(s))\left(\frac1{L}-\frac1{L_0}\right)\right]\\
=\frac{(\pa_s d_x\psi(s,x_0))(x-x_0)/\e}L-\pa_s(\psi(s,x_0)-z(s))\frac1{L_0}+O(\e).
\end{split}
\end{equation}
Using \eqref{x-curve} gives
\begin{equation}\label{x-limit}
x-x_0=\e d_x\psi(s,x_0)^{-1}(\psi(s,x_0)-z(s))+O(\e^2).
\end{equation}
Substitute \eqref{x-limit} into \eqref{step-1} and take the limit as $\e\to0$:
\begin{equation}\label{step-2}
(\pa_s d_x\psi(s,x_0))d_x\psi(s,x_0)^{-1}\bt(s,x_0)-\pa_s\bt(s,x_0)+c \bt(s,x_0)
\end{equation}
for some scalar $c$. Clearly,
\begin{equation}\label{play-beta}
\begin{split}
\pa_s \bt(s,x_0)&=\pa_s  \left(d_x\psi(s,x_0) d_x\psi(s,x_0)^{-1}\bt(s,x_0)\right)=\pa_s  \left(d_x\psi(s,x_0) \dot\ga(s,x_0)\right)\\
&=(\pa_s  d_x\psi(s,x_0)) \dot\ga(s,x_0)+d_x\psi(s,x_0) \pa_s\dot\ga(s,x_0).
\end{split}
\end{equation}
Recall that $\dot\ga$ is defined in \eqref{gadot-def}. Using \eqref{play-beta} simplifies \eqref{step-2} to
\begin{equation}\label{step-3}
-d_x\psi(s,x_0) \pa_s\dot\ga(s,x_0)+c \bt(s,x_0).
\end{equation}
Thus, in the limit as $x\to x_0$, the second condition in \eqref{cond-3-final} becomes
\begin{equation}\label{step-final}
\xi\cdot \pa_s\dot\ga(s,x_0)=0,
\end{equation}
where we have used \eqref{sj-def} again. Combining with \eqref{s-eqn} and comparing with \eqref{denom} proves the desired assertion.

\bibliographystyle{amsalpha}
\bibliography{bibliogr}

\end{document}